\documentclass[preprint,12pt,3p]{elsarticle}
\usepackage{amsmath,amsthm,amssymb,amstext,amsfonts,amscd,mathrsfs}
\usepackage{lineno,hyperref}
\modulolinenumbers[5]
\journal{....}
\usepackage{float}
\usepackage{graphicx}
\usepackage{subfigure}
\usepackage{placeins}
\newtheorem{theorem}{Theorem}
\theoremstyle{plain}

\newtheorem{lemma}{Lemma}

\numberwithin{equation}{section}

\begin{document}

\begin{frontmatter}

\title{On Generalized Fractional Derivatives Involving Generalized k-Mittag Leffler Function}

\author[mymainaddress1]{Mehar Chand}
\ead{mehar.jallandhra@gmail.com}

\author[mymainaddress2]{Jatinder Kumar Bansal}
\ead{jatinderbansalpatran@gmail.com}

\address[mymainaddress1]{Department of Mathematics, Baba Farid College, Bathinda-151001 (India)}

\address[mymainaddress2]{Department of Applied Sciences, Guru Kashi University, Bathinda-151002 (India)}

\begin{abstract}
In this paper, certain generalized fractional derivative formulae are introduced involving
the k-Mittag-Leffler function. Then their image formulae (using Beta transform, Laplace
transform and Whittaker transform) are also established. The results obtained here are quite
general in nature. The special cases of our findings are also discussed.
\end{abstract}

\end{frontmatter}

%\linenumbers

\section{Introduction}

%We shall start this section with some historical remarks and definitions to refresh our memories about some of the remarkable milestones in the theory of fractional calculus. As is well known now-a-days,
%the first document note about a fractional derivative was found in 1695 in the letters of Lebnitz to
%L'Hospital \cite{leibniz2, leibniz1, leibniz, Ross}.  In 1819, Lacroix, obtained the well known $ \frac{1}{2} $ derivative of x \cite{lacroix}, using inductive
%arguments to be $ d^ \frac{1}{2} / dx^\frac{1}{2} = 2\sqrt{x/\pi} $ long before the Riemann-Liouville fractional derivative surfaced into
%the realm of fractional calculus. The idea of a derivative that is not of an order of a positive integer was
%introduced by Liouville in 1832 \cite{liouville, lutzen, osler}, in a manner that would generalize the relation  $ D^n e^{\xi x} = \xi^n e^{\xi x}  $
%to any complex number n. Liouville then used Fourier theory to extend this $ \xi $ derivative to any function
%f(z) expanded in a Fourier series \cite{osler}. In 1888, Nekrassov \cite{nekr, osler1, osler}, generalizing the Cauchy's integral
%formula,
%$$\frac{d^n f(z)}{dz^n} = \frac{n!}{2\pi i} \oint_C \frac{f(\zeta)}{(\zeta - z)^{n+1}}d\zeta,$$
%where C is a closed contour surrounding the point z and enclosing a region of analyticity of f, came up
%with a fractional derivative which can be showed to be equal to Riemann-Liouville derivative under certain
%conditions \cite{oldham}.
%\newline

Diaz and Pariguan \cite{DiPa07}  introduced the $k$-Pochhammer symbol  and $k$-gamma function defined as follows:

\begin{eqnarray}\label{P1}
\aligned & (\vartheta)_{n,k}:
  =\left\{\aligned & \frac{\Gamma_k(\vartheta+nk)}{\Gamma_k(\vartheta)}  \hskip 21 mm (k\in\mathbb{R};\vartheta \in \mathbb{C}\setminus \{0\}) \\ & \vartheta(\vartheta+k)...(\vartheta+(n-1)k) \hskip 10mm (n\in\mathbb{N};\vartheta \in \mathbb{C}),
\endaligned \right.
\endaligned
\end{eqnarray}

and the relation with the classical Euler's gamma function as:

\begin{eqnarray}\label{P2}
\aligned \Gamma_k(\vartheta)=k^{\frac{\vartheta}{k}-1}\Gamma\left(\frac{\vartheta}{k}\right),
\endaligned
\end{eqnarray}

where $\vartheta \in\mathbb{C}, k \in\mathbb{R}$ and $n \in\mathbb{N}.$

When $ k=1 $, \eqref{P1} reduces to the classical  Pochhammer symbol and Euler's gamma function respectively.

Also let $\vartheta\in\mathbb{C}, k,s \in\mathbb{R}$, then the following identity holds

\begin{eqnarray}\label{P3}\aligned & \Gamma_s(\vartheta)
  =\left(\frac{s}{k}\right)^{\frac{\vartheta}{s}-1}\Gamma_k\left(\frac{k\vartheta}{s}\right),\endaligned\end{eqnarray}

 in particular,

\begin{equation}\label{P4}
\aligned \Gamma_k(\vartheta)=k^{\frac{\vartheta}{k}-1}\Gamma\left(\frac{\vartheta}{k}\right).
\endaligned
\end{equation}

Further, let $\vartheta\in\mathbb{C}, k,s\in\mathbb{R}$ and $\vartheta\in\mathbb{C}$, then the following identity holds

\begin{equation}\label{P5}\aligned & (\vartheta)_{nq,s}=\left(\frac{s}{k}\right)^{nq}\left(\frac{k\vartheta}{s}\right)_{nq},\endaligned
\end{equation}

 in particular,

\begin{equation}\label{P6}
\aligned& (\vartheta)_{nq,k}=(k)^{nq}\left(\frac{\vartheta}{k}\right)_{nq}.
\endaligned
\end{equation}

 For more details of k-Pochhammer symbol, k-special function and fractional Fourier transform one can refer to the papers by Romero et. al.\cite{RoCe12, RoCe11}.

\vskip 3mm

Let $k\in \mathbb{R}, \xi, \zeta, \vartheta\in \mathbb{C}; \Re(\xi)>0, \Re(\zeta)>0, \Re{(\vartheta)}>0$ and $ q\in\mathbb{R}^+$, then the generalized k-Mittag-Leffler function, denoted by $ E^{\vartheta,q}_{k,\xi,\zeta}(z)$, is defined as

\begin{equation}\label{k-ML}
\aligned& E^{\vartheta,q}_{k,\xi,\zeta}(z)=\sum^\infty_{n=0}\frac{(\vartheta)_{nq,k} z^n}{\Gamma_k(n \xi+\zeta)n!},
\endaligned
\end{equation}

where $(\vartheta)_{nq,k}$ denotes the k-Pochhammer symbol given by equation \eqref{P6} and $\Gamma_k(\vartheta)$ is the k-gamma function given by the equation \eqref{P4} as (also see\cite{SrTo09}).
\newline

Particular cases of $E^{\vartheta,q}_{k,\xi,\zeta}(z):$
\newline
%%%%%%%%%%%%%%%%%%%%%%%%%%%%%%%%%%%%%%%%%%%%%%%%%%%%%%%%%%%%%%%%%%%%%%%%%%%%%%%%%%%%%%%%%%%%%%%%%%%%%%%%%%%%%%%%%%%%%%%%%%%%%%%%%%%%%

For $q=1$, equation\eqref{k-ML} yields  k-Mittag-Leffler function defined as:

\begin{equation}\label{kML1}
\aligned& E^{\vartheta,1}_{k,\xi,\zeta}(z)=\sum^\infty_{n=0}\frac{(\vartheta)_{n,k} z^n}{\Gamma_k(n \xi+\zeta)n!}=E^{\vartheta}_{k,\xi,\zeta}(z).
\endaligned
\end{equation} 

For $k=1$, equation\eqref{k-ML} yields Mittag-Leffler function, defined as (Shukla and Prajapati \cite{ShPr07})

\begin{equation}\label{kML2}
\aligned & E^{\vartheta,q}_{1,\xi,\zeta}(z)=\sum^\infty_{n=0}\frac{(\vartheta)_{nq} z^n}{\Gamma(n \xi+\zeta)n!}=E^{\vartheta,q}_{\xi,\zeta}(z).
\endaligned
\end{equation}

For $q=1$ and $k=1$, equation\eqref{k-ML} gives Mittag-Leffler function, defined as

\begin{equation}\label{kML3}
\aligned& E^{\vartheta,1}_{1,\xi,\zeta}(z)=\sum^\infty_{n=0}\frac{(\vartheta)_{n} z^n}{\Gamma(n \xi+\zeta)n!}=E^{\vartheta}_{\xi,\zeta}(z).
\endaligned
\end{equation}

For $q=1, k=1$ and $\vartheta=1$, equation\eqref{k-ML} gives Mittag-Leffler function (Wiman \cite{Wi05}), defined as

\begin{equation}\label{kML4}
\aligned& E^{1,1}_{1,\xi,\zeta}(z)=\sum^\infty_{n=0}\frac{ z^n}{\Gamma(n \xi+\zeta)}=E_{\xi,\zeta}(z).
\endaligned
\end{equation}

For $q=1, k=1, \vartheta=1$ and $\zeta=1$, equation\eqref{k-ML} gives Mittag-Leffler function is defined as

\begin{equation}\label{kML5}
\aligned & E^{1,1}_{1,\xi,1}(z)=\sum^\infty_{n=0}\frac{ z^n}{\Gamma(n \xi+1)}=E_{\xi}(z).
\endaligned
\end{equation}
\newline

%%%%%%%%%%%%%%%%%%%%%%%%%%%%%%%%%%%%%%%%%%%%%%%%%%%%%%%%%%%%%%%%%%%%%%%%%%%%%%%%%%%%%%%%%%%%%%%%%%%%%%%%%%%%%%%%%%%%%%%%%%%%%%

The Fox-Wright function ${}_p\Psi_q[z]$ defined as 

\begin{eqnarray} \label{FW}\aligned &{}_p \Psi_q[z]= {}_p \Psi_q\left[\begin{array}{cc}(a_1,\xi_1),..., (a_p,\xi_p);\\
(b_1,\zeta_1),..., (b_q,\zeta_q);\end{array} z\right]\\&\hskip 4mm={}_p \Psi_q\left[\begin{array}{cc}(a_i,\xi_i)_{1,p};\\
(b_j,\zeta_j)_{1,q};\end{array} z\right] =\sum^{\infty}_{n=0}\frac{\prod^{p}_{i=1}\Gamma(a_i+\xi_i n)}{\prod^{q}_{j=1}\Gamma(b_j+\zeta_j n)}\frac{z^n}{n!}, \endaligned \end{eqnarray}

where the coefficients $\xi_1,...,\xi_p$, $\zeta_1,...,\zeta_q$ $\in \mathbb{R}^+$ such that

\begin{eqnarray} \aligned & 1+\sum^q_{j=1}\zeta_j-\sum^p_{i=1}\xi_i\geq 0.\endaligned \end{eqnarray}

%%%%%%%%%%%%%%%%%%%%%%%%%%%%%%%%%%%%%%%%%%%%%%%%%%%%%%%% Fractional Integration %%%%%%%%%%%%%%%%%%%%%%%%%%%%%%%%%%%%%%%%%%%%%%%%%%%%%

\section{Fractional integration}

 In this section, we will establish some fractional integral formulas
for the generalized k-Mittag-Leffler function. To do this, we need to recall the following pair of fractional integral operators.
\newline

The Riemann-Liouville fractional integrals $I^\xi _{a+}f$ and $I^\xi _{b-}f$ of order $\xi \in \mathbb{C},\Re(\xi)>0$, are defined
by \cite{Kilbas06, liouville, Riemann, Riesz, samkokilbas},
\begin{eqnarray}\label{qw}
(I^\xi _{a+}f)(x)=\frac{1}{\Gamma(\xi)} \int^x _a (x - \tau)^{\xi-1}f(\tau)d\tau ; \hskip5mm x>a
\end{eqnarray}
and
\begin{eqnarray}\label{qwe}
(I^\xi _{b-}f)(x)=\frac{1}{\Gamma(\xi)} \int^b _x (\tau-x)^{\xi-1}f(\tau)d\tau ; \hskip5mm x<b,
\end{eqnarray} 
respectively. Here $\tau(.) $ is the Gamma function. These integrals are called the left-sided and right-sided
fractional integrals, respectively. When $\xi=n\in \mathbb{N}$, the integrals (2.1) and (2.2) coincide with the n-fold
integrals \cite{Kilbas06}.
\begin{lemma}\label{lemma-1}Let $\Omega=[a,b](-\infty<a<b<\infty)$ be a finite interval on the real axis $\mathbb{R}$. The generalized fractional integral ${}^\eta I_{a+}^\sigma f$ of order $\sigma\in\mathbb{C}$ for $x>a$ and $\Re(\sigma)>0$ is defined as

 \begin{eqnarray}\label{Ia}\aligned &\left({}^{\eta}I^{\sigma}_{a+}f\right)(x)=\frac{(\eta)^{1-\sigma}}{\Gamma(\sigma)}\int^{x}_{a}\frac{t^{\eta-1}f(t)}{(x^{\eta}-t^{\eta})^{1-\sigma}}dt, \endaligned \end{eqnarray}

similarly generalized fractional integral ${}^\eta I_{b-}^\sigma f$ of order $\sigma\in\mathbb{C}$ for $x<b$ and $\Re(\sigma)>0$ is defined as

 \begin{eqnarray}\label{Ib}\aligned &\left({}^{\eta}I^{\sigma}_{b-}f\right)(x)=\frac{(\eta)^{1-\sigma}}{\Gamma(\sigma)}\int^{b}_{x}\frac{t^{\eta-1}f(t)}{(t^{\eta}-x^{\eta})^{1-\sigma}}dt. \endaligned \end{eqnarray}
\end{lemma}

If we choose $a=b=0$ the above Lemma \ref{lemma-1} reduces to

\begin{lemma}\label{lemma-2} The generalized fractional integral ${}^\eta I_{0+}^\sigma f$ of order $\sigma\in\mathbb{C}$ for $x>0$ and $\Re(\sigma)>0$ is defined as

 \begin{eqnarray}\label{I1}\aligned &\left({}^{\eta}I^{\sigma}_{0+}f\right)(x)=\frac{(\eta)^{1-\sigma}}{\Gamma(\sigma)}\int^{x}_{0}\frac{t^{\eta-1}f(t)}{(x^{\eta}-t^{\eta})^{1-\sigma}}dt, \endaligned \end{eqnarray}

similarly generalized fractional integral ${}^\eta I_{0-}^\sigma f$ of order $\sigma\in\mathbb{C}$ for $x<0$ and $\Re(\sigma)>0$ is defined as

 \begin{eqnarray}\label{I2}\aligned &\left({}^{\eta}I^{\sigma}_{0-}f\right)(x)=\frac{(\eta)^{1-\sigma}}{\Gamma(\sigma)}\int^{0}_{x}\frac{t^{\eta-1}f(t)}{(t^{\eta}-x^{\eta})^{1-\sigma}}dt. \endaligned \end{eqnarray}
 \end{lemma}

\begin{lemma}\label{lemma-3} Riemann-type fractional derivatives ${}^{\eta}D^{\sigma}_{a+}f$ and ${}^{\eta}D^{\sigma}_{b-}f$ of order $\sigma\in\mathbb{C},\Re(\sigma)>0$ are defined as

\begin{eqnarray}\label{d1}\aligned &\left({}^{\eta}D^{\sigma}_{a+}f(t)\right)(x)=\left(x^{1-\eta}\frac{d}{dx}\right)^n \left({}^{\eta}I^{n-\sigma}_{a+}f(t)\right)(x)\\&\hskip 20mm = \frac{{\eta}^{\sigma-n+1}}{\Gamma(n-\xi)}\left(x^{1-\eta}\frac{d}{dx}\right)^n\int^{x}_{a}\frac{t^{\eta-1}f(t)}{(x^{\eta}-t^{\eta})^{\sigma-n+1}}dt \hskip 10mm for \hskip 2mm   x>a \endaligned 
\end{eqnarray}

and 

\begin{eqnarray}\label{d2}\aligned &\left({}^{\eta}D^{\sigma}_{b-}f(t)\right)(x)=\left(-x^{1-\eta}\frac{d}{dx}\right)^n \left({}^{\eta}I^{n-\sigma}_{b-}f(t)\right)(x)\\&\hskip 20mm = \frac{{\eta}^{\sigma-n+1}}{\Gamma(n-\xi)}\left(-x^{1-\eta}\frac{d}{dx}\right)^n\int^{b}_{x}\frac{t^{\eta-1}f(t)}{(t^{\eta}-x^{\eta})^{\sigma-n+1}}dt \hskip 10mm for \hskip 2mm   x<b, \endaligned 
\end{eqnarray}

where $n=[\Re(\sigma)]+1.$
\end{lemma}

\begin{eqnarray}\label{O4a}\aligned & \frac{d^n}{dx^n}(x^{\sigma})=\frac{\Gamma(\sigma+1)}{\Gamma(\sigma+1-n)}x^{\sigma-n},\hskip 5mm \Re(\sigma)>0.\endaligned \end{eqnarray} 

If we choose $a=b=0$ the above Lemma \ref{lemma-3} reduces to 

\begin{lemma}\label{l-4} The generalized fractional integral ${}^\eta I_{0+}^\sigma f$ of order $\sigma\in\mathbb{C}$ for $x>0$ and $\Re(\sigma)>0$ is defined as

\begin{eqnarray}\label{d1}\aligned &\left({}^{\eta}D^{\sigma}_{0+}f(t)\right)(x)=\left(x^{1-\eta}\frac{d}{dx}\right)^n \left({}^{\eta}I^{n-\sigma}_{0+}f(t)\right)(x)= \frac{{\eta}^{\sigma-n+1}}{\Gamma(n-\xi)}\left(x^{1-\eta}\frac{d}{dx}\right)^n\int^{x}_{0}\frac{t^{\eta-1}f(t)}{(x^{\eta}-t^{\eta})^{\sigma-n+1}}dt, \endaligned 
\end{eqnarray}

similarly generalized fractional integral ${}^\eta I_{0-}^\sigma f$ of order $\sigma\in\mathbb{C}$ for $x<0$ and $\Re(\sigma)>0$ is defined as

\begin{eqnarray}\label{d2}\aligned &\left({}^{\eta}D^{\sigma}_{0-}f(t)\right)(x)=\left(-x^{1-\eta}\frac{d}{dx}\right)^n \left({}^{\eta}I^{n-\sigma}_{0-}f(t)\right)(x)= \frac{{\eta}^{\sigma-n+1}}{\Gamma(n-\xi)}\left(-x^{1-\eta}\frac{d}{dx}\right)^n\int^{0}_{x}\frac{t^{\eta-1}f(t)}{(t^{\eta}-x^{\eta})^{\sigma-n+1}}dt. \endaligned 
\end{eqnarray}

The main results are given in the following theorem.
\end{lemma}

\begin{theorem}\label{TH1} Let $x>0, \sigma, \eta,\mu, \xi,\zeta,\vartheta  \in \mathbb{C}, k \in\mathbb{R}$, $\Re(\sigma)>0, \Re(\zeta)>0, \Re(\vartheta)>0$ and $q\in\mathbb{R}^+$, such that  $\eta>0$, then

\begin{eqnarray}\label{TH1-1}\aligned& \left({}^{\eta}D^{\sigma}_{0+}t^\mu E^{\vartheta, q}_{k,\xi,\zeta}(t^\nu)\right)(x)=x^{\mu-\sigma\eta}\frac{{\eta}^\sigma}{\Gamma\left(\frac{\vartheta}{k}\right)}k^{1-\frac{\zeta}{k}}{}_{2}\Psi_{2}\left[
 \begin{array}{cc} 
\left(\frac{\vartheta}{k},{q}\right), \left(\frac{\mu}{\eta}+1,\frac{\nu}{\eta}\right)\\
	\left(\frac{\zeta}{k},\frac{\xi}{k}\right), \left(\frac{\mu}{\eta}-\sigma+1,\frac{\nu}{\eta}\right) \end{array}\left|k^{(q-\frac{\xi}{k})}x^{\nu}\right.
\right]. \endaligned \end{eqnarray}
\end{theorem}

\begin{proof} For convenience, we denote the left-hand side of the result \eqref{TH1-1} by $\mathscr{D}$. Using \eqref{k-ML}, and then changing the order of integration and summation, then

\begin{eqnarray}\label{TH1-2} \aligned& \mathscr{D}=\sum_{r=0}^{\infty}\frac{(\vartheta)_{rq,k}}{\Gamma_{k}(r\xi+\zeta)}\frac{1}{r!}\left({}^{\eta}D^{\sigma}_{0+}t^{r\nu+\mu}\right), \endaligned \end{eqnarray}

applying the result \eqref{d1}, the above equation \eqref{TH1-2} reduced to

\begin{eqnarray}\label{TH1-3}\aligned& \mathscr{D}=\sum_{r=0}^{\infty}\frac{(\vartheta)_{rq,k}}{\Gamma_{k}(r\xi+\zeta)}\frac{1}{r!} \frac{{\eta}^{\sigma}}{\Gamma(1-\sigma)}\left(x^{1-\eta}\frac{d}{dx}\right)\int^{x}_{0}\frac{t^{\eta-1}t^{r\nu+\mu}}{(x^{\eta}-t^{\eta})^{\sigma}}dt. \endaligned \end{eqnarray}

Put $t^{\eta}=x^{\eta}z$ in equation \eqref{TH1-3} and by proper substitution we have

\begin{eqnarray}\label{TH1-4}\aligned& \mathscr{D}=\sum_{r=0}^{\infty}\frac{{\eta}^{\sigma}}{\Gamma(1-\sigma)}k^{1-\frac{\zeta}{k}}\frac{1}{r!}\frac{\Gamma(\frac{\vartheta}{k}+rq)}{\Gamma(\frac{\vartheta}{k})\Gamma(\frac{\zeta}{k}+\frac{\xi}{k}r)} k^{r(q-\frac{\xi}{k})}\frac{1}{\eta}\left(x^{1-\eta}\frac{d}{dx}\right) \\&\hskip 10mm \times\left( x^{\eta+r\nu+\mu-\sigma\eta} \int_{0}^{1} z^{\frac{r \nu+\mu}{\eta}}(1-z)^{-\sigma}dz\right). \endaligned \end{eqnarray}

\begin{eqnarray}\label{TH1-5}\aligned& \mathscr{D}=\sum_{r=0}^{\infty}\frac{{\eta}^{\sigma}}{\Gamma(1-\sigma)}k^{1-\frac{\zeta}{k}}\frac{1}{r!}\frac{\Gamma(\frac{\vartheta}{k}+rq)}{\Gamma(\frac{\vartheta}{k})\Gamma(\frac{\zeta}{k}+\frac{\xi}{k}r)} k^{r(q-\frac{\xi}{k})} \times\left( x^{r\nu+\mu-\sigma\eta}  \frac{\Gamma\left(\frac{\mu}{\eta}+1+\frac{\nu}{\eta}r\right)\Gamma(1-\sigma)}{\Gamma\left(\frac{\mu}{\eta}+1-\sigma+\frac{\nu}{\eta}r\right)}\right), \endaligned \end{eqnarray}

after simplification, the above equation \eqref{TH1-5} reduces to

\begin{eqnarray}\label{TH1-6}\aligned& \mathscr{D}=x^{\mu-\sigma\eta}\frac{{\eta}^\sigma}{\Gamma\left(\frac{\vartheta}{k}\right)}k^{1-\frac{\zeta}{k}}\sum_{r=0}^{\infty}\frac{1}{r!}\frac{\Gamma(\frac{\vartheta}{k}+rq)}{\Gamma(\frac{\zeta}{k}+\frac{\xi}{k}r)} \frac{\Gamma\left(\frac{\mu}{\eta}+1+\frac{\nu}{\eta}r\right)}{\Gamma\left(\frac{\mu}{\eta}+1-\sigma+\frac{\nu}{\eta}r\right)}  k^{r(q-\frac{\xi}{k})}.\endaligned \end{eqnarray}

By using equation \eqref{TH1-6} and simplification, we have

\begin{eqnarray}\label{TH1-7}\aligned& \left({}^{\eta}D^{\sigma}_{0+}t^\mu E^{\vartheta, q}_{k,\xi,\zeta}(t^\nu)\right)(x)=x^{\mu-\sigma\eta}\frac{{\eta}^\sigma}{\Gamma\left(\frac{\vartheta}{k}\right)}k^{1-\frac{\zeta}{k}}{}_{2}\Psi_{2}\left[
 \begin{array}{cc} 
\left(\frac{\vartheta}{k},{q}\right), \left(\frac{\mu}{\eta}+1,\frac{\nu}{\eta}\right)\\
	\left(\frac{\zeta}{k},\frac{\xi}{k}\right), \left(\frac{\mu}{\eta}-\sigma+1,\frac{\nu}{\eta}\right) \end{array}\left|k^{(q-\frac{\xi}{k})}x^{\nu}\right.
\right]. \endaligned \end{eqnarray}

 \end{proof}

\begin{theorem}\label{TH2} Let $x>0, \sigma,\mu, \eta, \xi,\zeta,\vartheta  \in \mathbb{C}, k \in\mathbb{R}$, $\Re(\sigma)>0, \Re(\zeta)>0, \Re(\vartheta)>0$ and $q\in\mathbb{R}^+$, such that  $\eta>0$, then

\begin{eqnarray}\label{TH2-1}\aligned& \left({}^{\eta}D^{\sigma}_{0-}t^\mu E^{\vartheta, q}_{k,\xi,\zeta}(t^\nu)\right)(x)=(-1)^{-\sigma} x^{\mu-\sigma\eta}\frac{{\eta}^\sigma}{\Gamma\left(\frac{\vartheta}{k}\right)}k^{1-\frac{\zeta}{k}}{}_{2}\Psi_{2}\left[
 \begin{array}{cc} 
\left(\frac{\vartheta}{k},{q}\right), \left(\frac{\mu}{\eta}+1,\frac{\nu}{\eta}\right)\\
	\left(\frac{\zeta}{k},\frac{\xi}{k}\right), \left(\frac{\mu}{\eta}-\sigma+1,\frac{\nu}{\eta}\right) \end{array}\left|k^{(q-\frac{\xi}{k})}x^{\nu}\right.
\right]. \endaligned \end{eqnarray}
\end{theorem}

\begin{proof} For convenience, we denote the left-hand side of the result \eqref{TH2-1} by $\mathscr{D}$. Using \eqref{k-ML}, and then changing the order of integration and summation, then

\begin{eqnarray}\label{TH2-2} \aligned& \mathscr{D}=\sum_{r=0}^{\infty}\frac{(\vartheta)_{rq,k}}{\Gamma_{k}(r\xi+\zeta)}\frac{1}{r!}\left({}^{\eta}D^{\sigma}_{0-}t^{r\nu+\mu}\right), \endaligned \end{eqnarray}

applying the result \eqref{d2}, the above equation \eqref{TH2-2} reduced to

\begin{eqnarray}\label{TH2-3}\aligned& \mathscr{D}=\sum_{r=0}^{\infty}\frac{(\vartheta)_{rq,k}}{\Gamma_{k}(r\xi+\zeta)}\frac{1}{r!} \frac{{\eta}^{\sigma}}{\Gamma(1-\sigma)}\left(-x^{1-\eta}\frac{d}{dx}\right)\int^{0}_{x}\frac{t^{\eta-1}t^{r\nu+\mu}}{(t^{\eta}-x^{\eta})^{\sigma}}dt. \endaligned \end{eqnarray}

Put $t^{\eta}=x^{\eta}z$ in equation \eqref{TH2-3} and by proper substitution we have

\begin{eqnarray}\label{TH2-4}\aligned& \mathscr{D}=(-1)^{-\sigma}\sum_{r=0}^{\infty}\frac{{\eta}^{\sigma}}{\Gamma(1-\sigma)}k^{1-\frac{\zeta}{k}}\frac{1}{r!}\frac{\Gamma(\frac{\vartheta}{k}+rq)}{\Gamma(\frac{\vartheta}{k})\Gamma(\frac{\zeta}{k}+\frac{\xi}{k}r)} k^{r(q-\frac{\xi}{k})}\frac{1}{\eta}\left(x^{1-\eta}\frac{d}{dx}\right) \\&\hskip 10mm \times\left( x^{\eta+r\nu+\mu-\sigma\eta} \int_{0}^{1} z^{\frac{r \nu+\mu}{\eta}}(1-z)^{-\sigma}dz\right). \endaligned \end{eqnarray}

\begin{eqnarray}\label{TH2-5}\aligned& \mathscr{D}=(-1)^{-\sigma}\sum_{r=0}^{\infty}\frac{{\eta}^{\sigma}}{\Gamma(1-\sigma)}k^{1-\frac{\zeta}{k}}\frac{1}{r!}\frac{\Gamma(\frac{\vartheta}{k}+rq)}{\Gamma(\frac{\vartheta}{k})\Gamma(\frac{\zeta}{k}+\frac{\xi}{k}r)} k^{r(q-\frac{\xi}{k})} \left( x^{r\nu+\mu-\sigma\eta}  \frac{\Gamma\left(\frac{\mu}{\eta}+1+\frac{\nu}{\eta}r\right)\Gamma(1-\sigma)}{\Gamma\left(\frac{\mu}{\eta}+1-\sigma+\frac{\nu}{\eta}r\right)}\right), \endaligned \end{eqnarray}

after simplification, the above equation \eqref{TH2-5} reduces to

\begin{eqnarray}\label{TH2-6}\aligned& \mathscr{D}=(-1)^{-\sigma}x^{\mu-\sigma\eta}\frac{{\eta}^\sigma}{\Gamma\left(\frac{\vartheta}{k}\right)}k^{1-\frac{\zeta}{k}}\sum_{r=0}^{\infty}\frac{1}{r!}\frac{\Gamma(\frac{\vartheta}{k}+rq)}{\Gamma(\frac{\zeta}{k}+\frac{\xi}{k}r)} \frac{\Gamma\left(\frac{\mu}{\eta}+1+\frac{\nu}{\eta}r\right)}{\Gamma\left(\frac{\mu}{\eta}+1-\sigma+\frac{\nu}{\eta}r\right)}  k^{r(q-\frac{\xi}{k})}.\endaligned \end{eqnarray}

By using equation \eqref{TH2-6} and simplification, we have

\begin{eqnarray}\label{TH2-7}\aligned& \left({}^{\eta}D^{\sigma}_{0+}t^\mu E^{\vartheta, q}_{k,\xi,\zeta}(t^\nu)\right)(x)=(-1)^{-\sigma}x^{\mu-\sigma\eta}\frac{{\eta}^\sigma}{\Gamma\left(\frac{\vartheta}{k}\right)}k^{1-\frac{\zeta}{k}}{}_{2}\Psi_{2}\left[
 \begin{array}{cc} 
\left(\frac{\vartheta}{k},{q}\right), \left(\frac{\mu}{\eta}+1,\frac{\nu}{\eta}\right)\\
	\left(\frac{\zeta}{k},\frac{\xi}{k}\right), \left(\frac{\mu}{\eta}-\sigma+1,\frac{\nu}{\eta}\right) \end{array}\left|k^{(q-\frac{\xi}{k})}x^{\nu}\right.
\right]. \endaligned \end{eqnarray}
\end{proof}

\subsection{Numerical results and graphical interpretation}
In this section we plot the graphs and obtained the numerical value of our findings in equation \eqref{TH1-1} and \eqref{TH2-1}. For this purpose, we select the values of the parameters involving in these results as $\mu=0.5; \nu=0.8; \eta=0.3; \xi=0.5;\zeta=0.2;\vartheta=0.5; q=0.4; k=0.5$ and $\sigma=0.1:0.1:0.4$ for Figure \ref{ig-1}. In Figure \ref{fig-2} the values of the figure are taken as $\eta=1:2:7;\sigma=0.02$
\newpage
\begin{figure}[h]
    \centering
       \fbox{\subfigure[Plot of Equation \eqref{TH2-1}]{\includegraphics[width=7.8 cm,height=7.6cm]{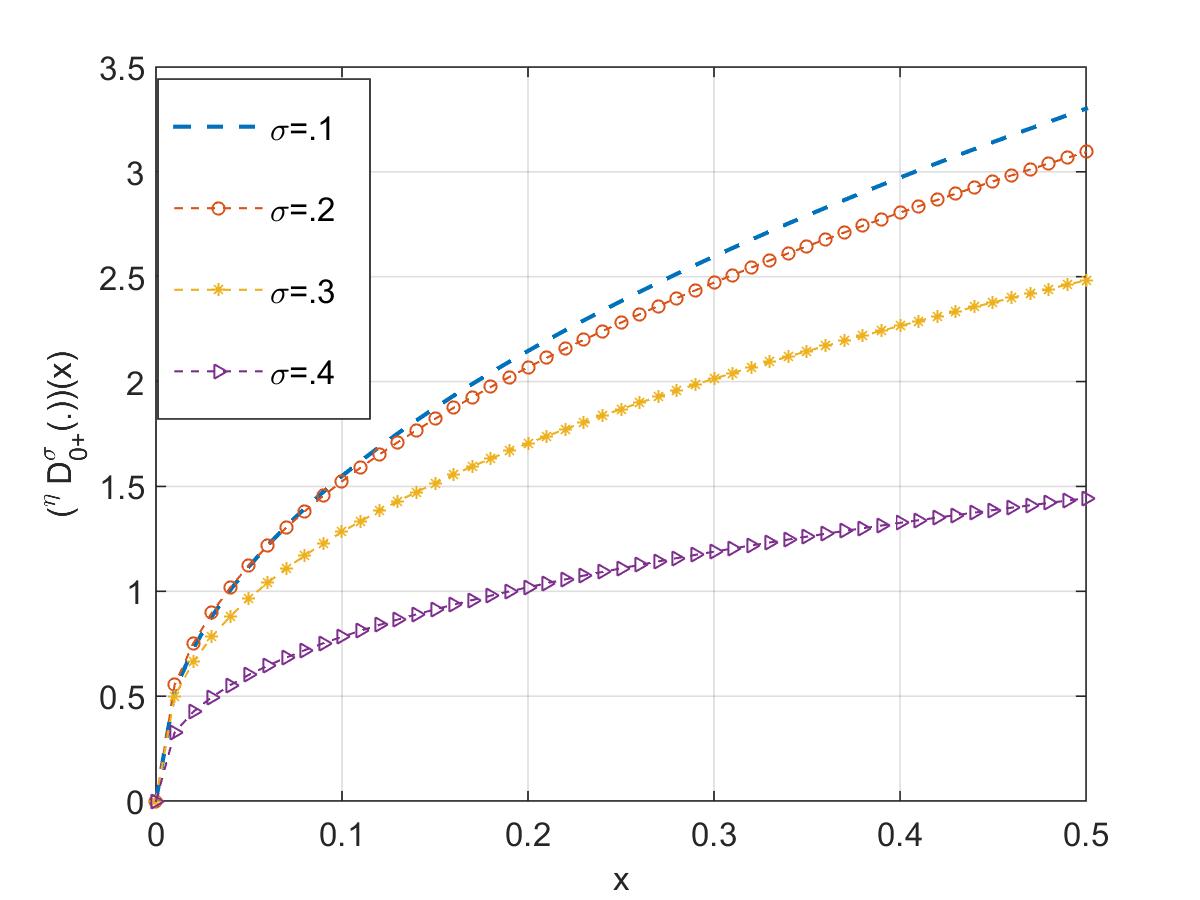}}
        \subfigure[Plot of Equation \eqref{TH1-1}]{\includegraphics[width=7.8 cm,height=7.6cm]{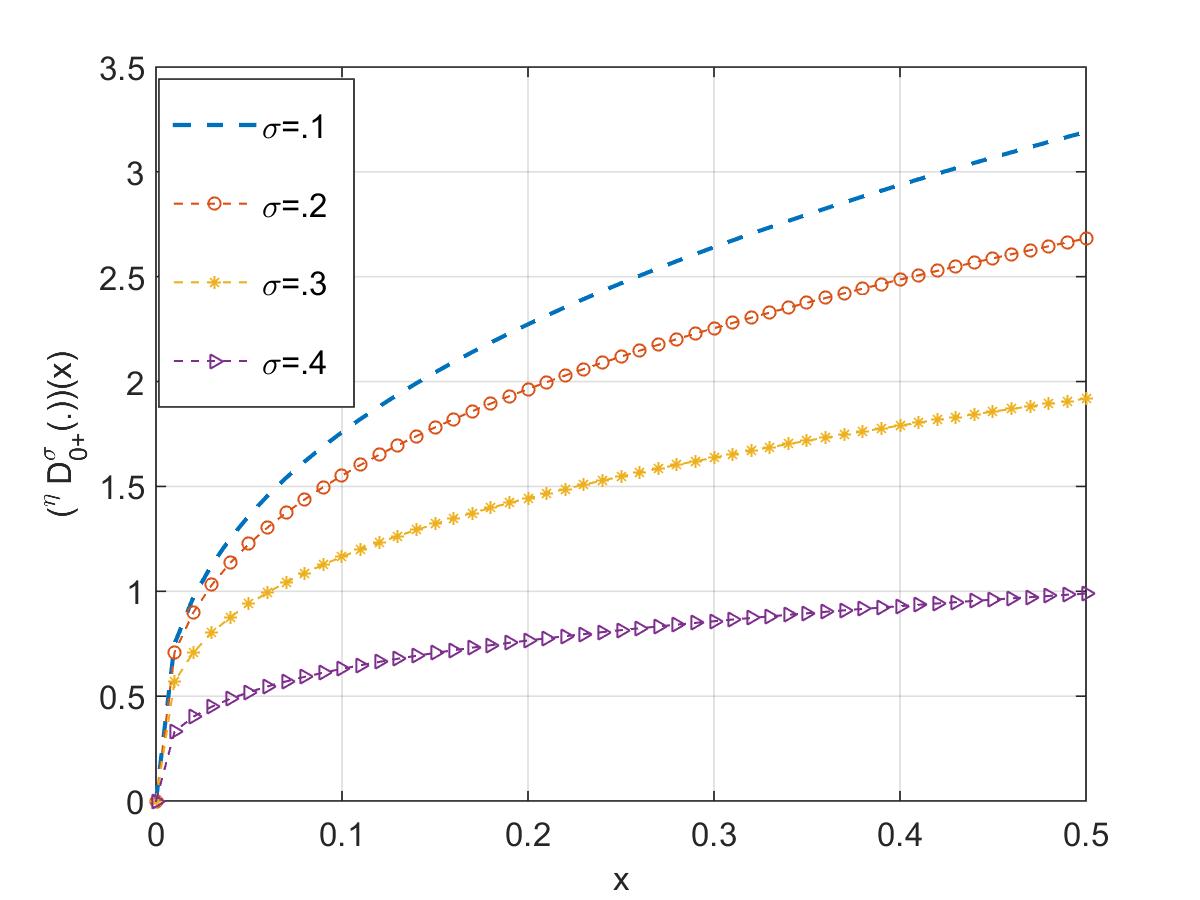}}}
        \caption{Graph for the values $\sigma=0.1:0.1:0.4;\eta=0.2$} \label{fig-1}
    \end{figure}

\begin{figure}[h]
    \centering
       \fbox{\subfigure[Plot of Equation \eqref{TH2-1}]{\includegraphics[width=7.8 cm,height=7.6cm]{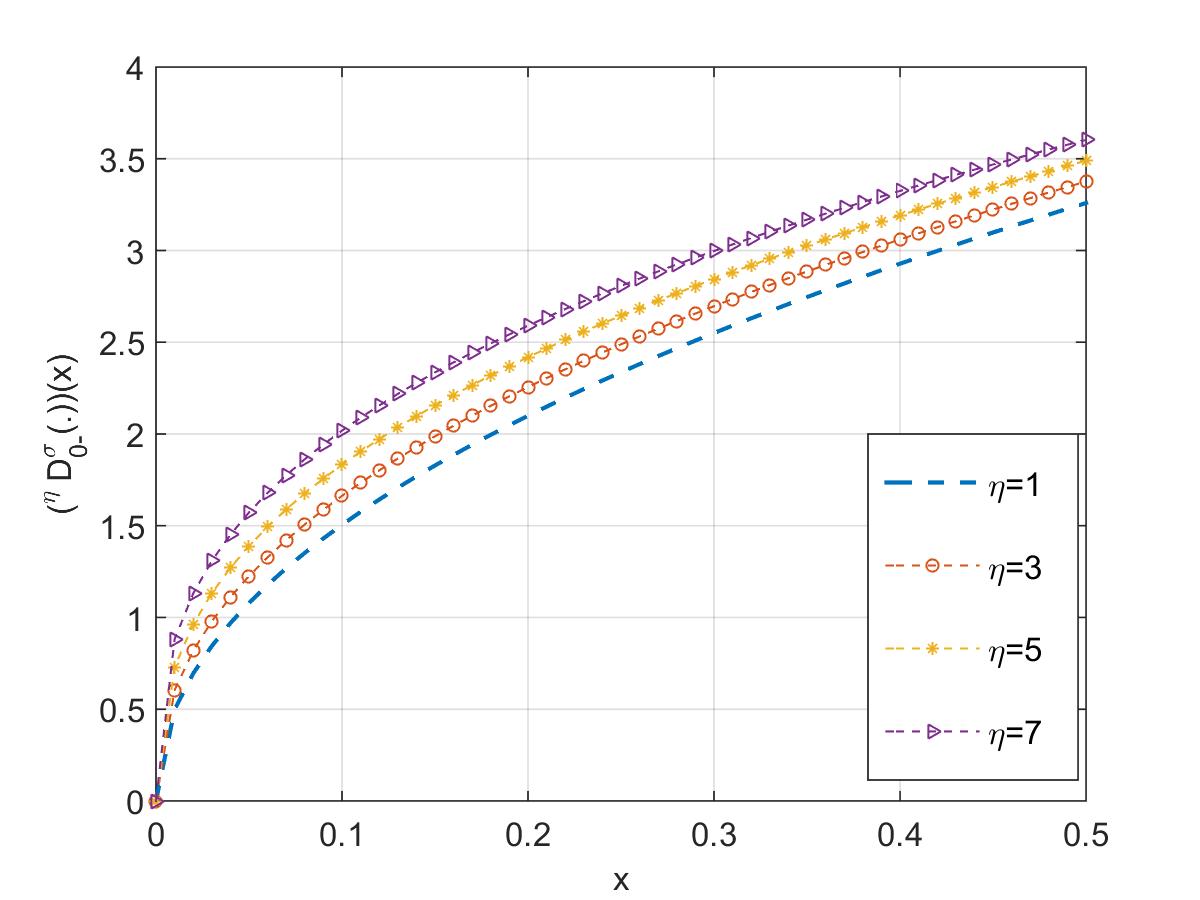}}
        \subfigure[Plot of Equation \eqref{TH1-1}]{\includegraphics[width=7.8 cm,height=7.6cm]{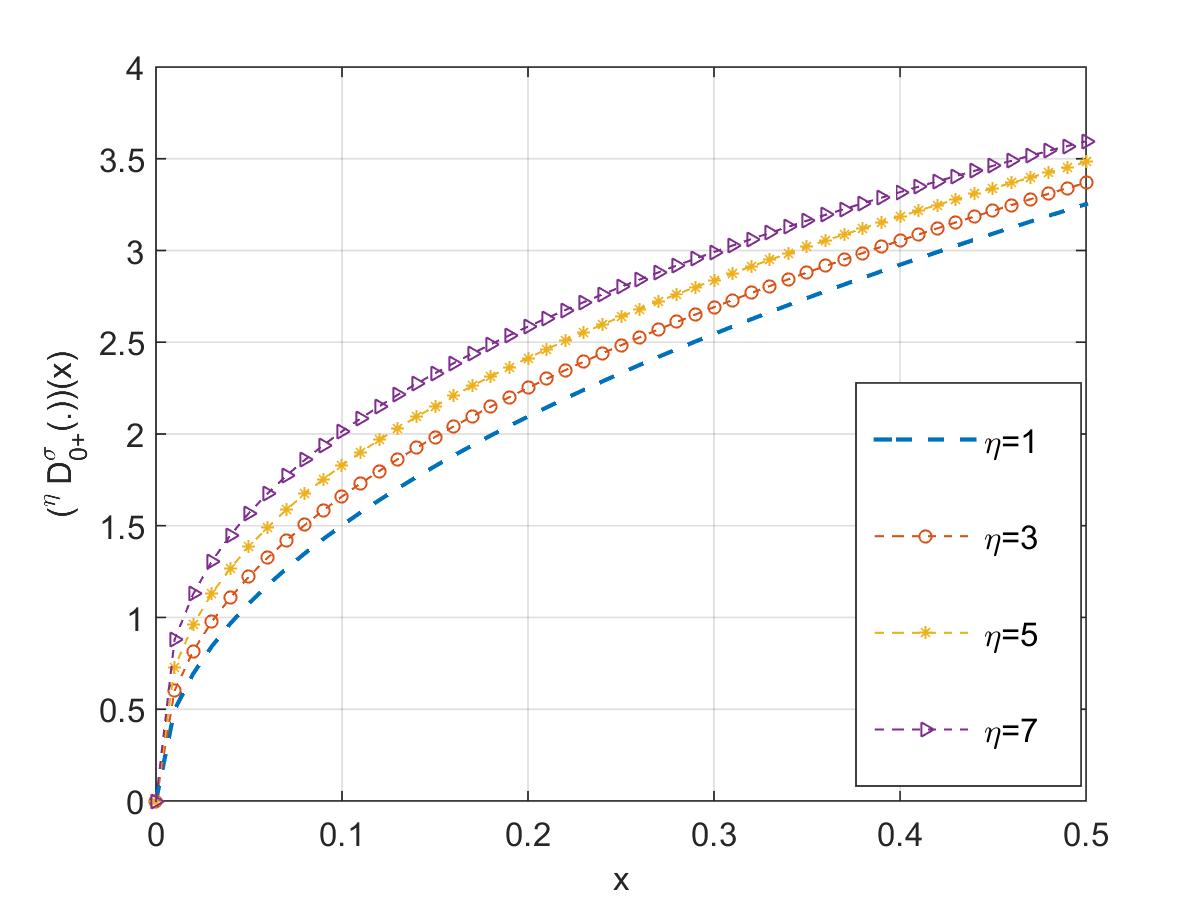}}}
        \caption{Graph for the values $\eta=1:2:7;\sigma=0.02$} \label{fig-2}
    \end{figure}

\begin{figure}[h]
    %\centering
       \fbox{\subfigure[Plot of Equation \eqref{TH2-1}]{\includegraphics[width=7.8 cm,height=7.6cm]{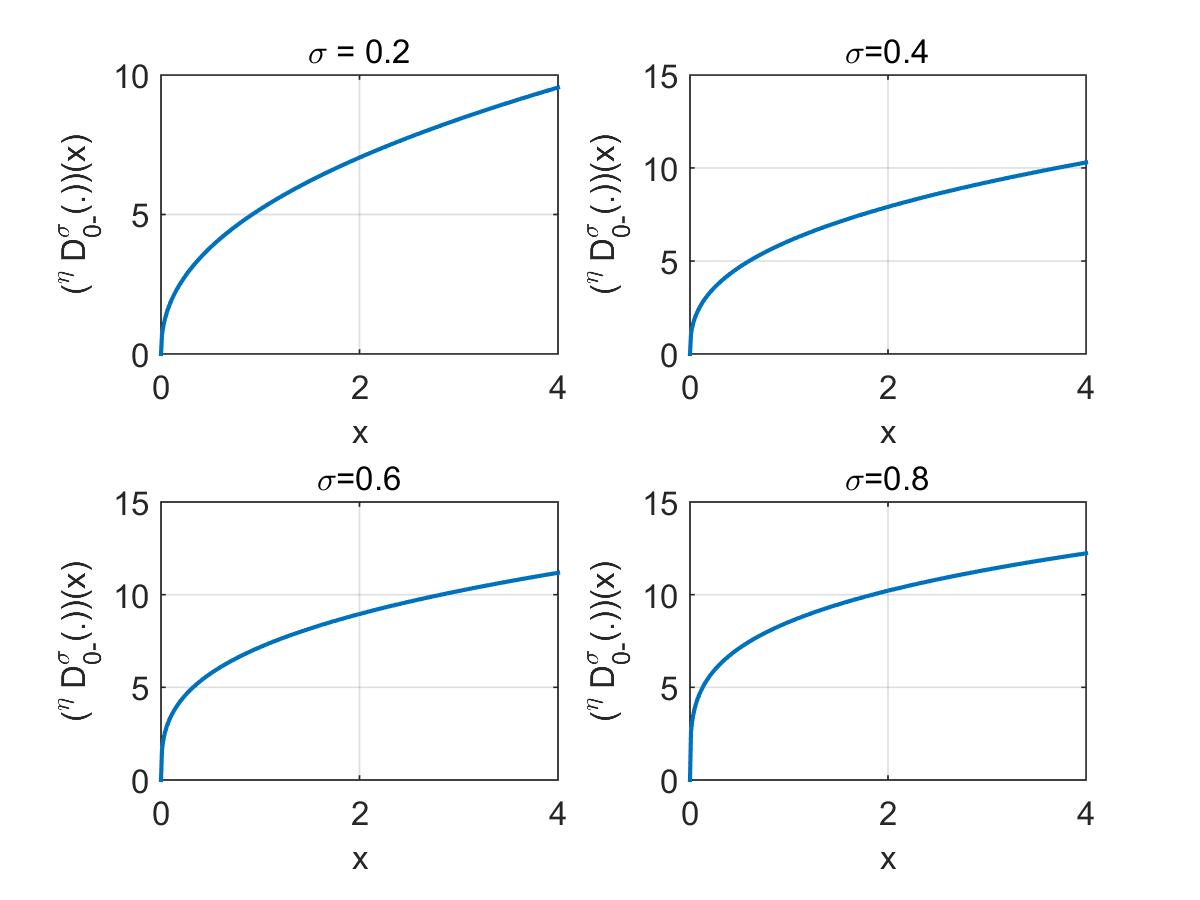}}\\
        \subfigure[Plot of Equation \eqref{TH1-1}]{\includegraphics[width=7.8 cm,height=7.6cm]{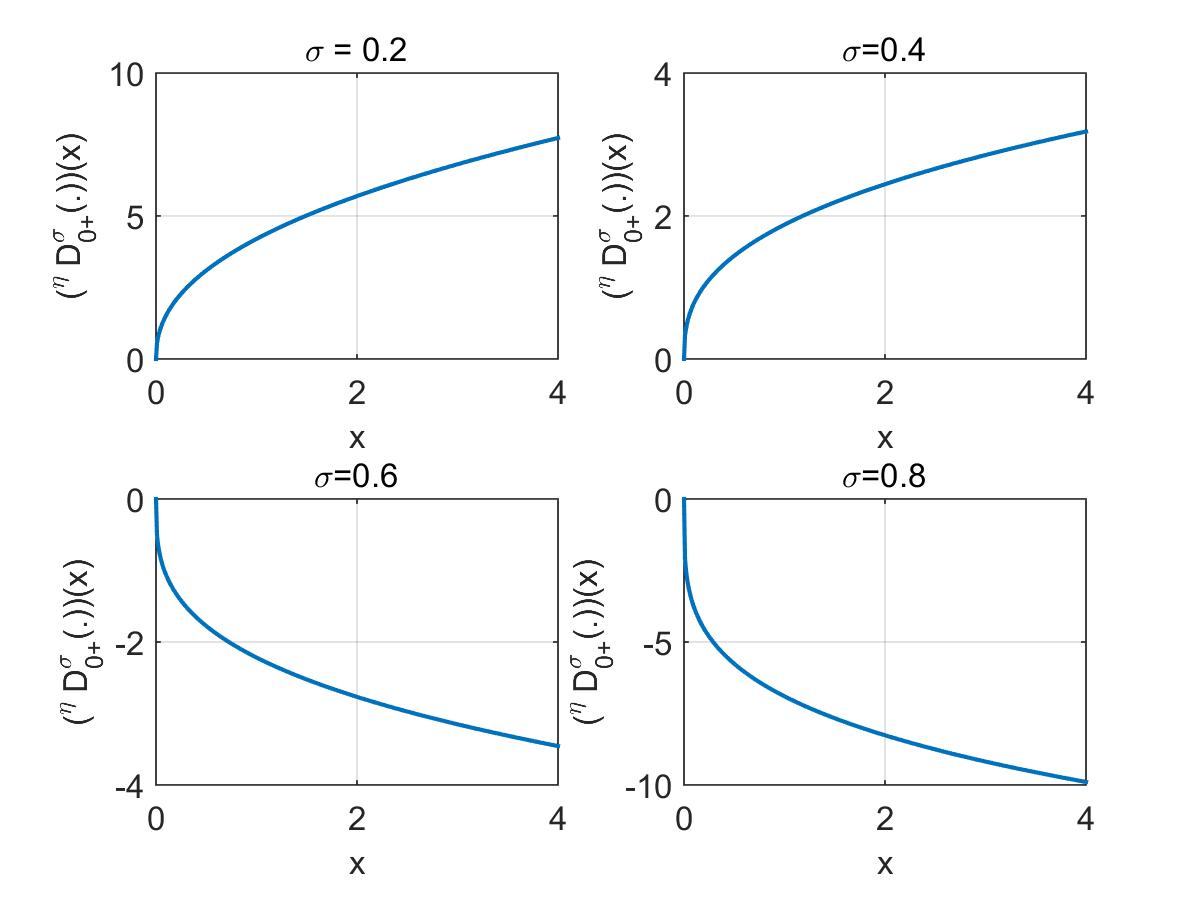}}}
        \caption{Graph  for $\sigma=0.02;\zeta=0.2$} \label{fig-3}
    \end{figure}

\begin{figure}[htb]
    \centering
       \fbox{\subfigure[Plot of Equation \eqref{TH2-1}]{\includegraphics[width=7.8 cm,height=7.6cm]{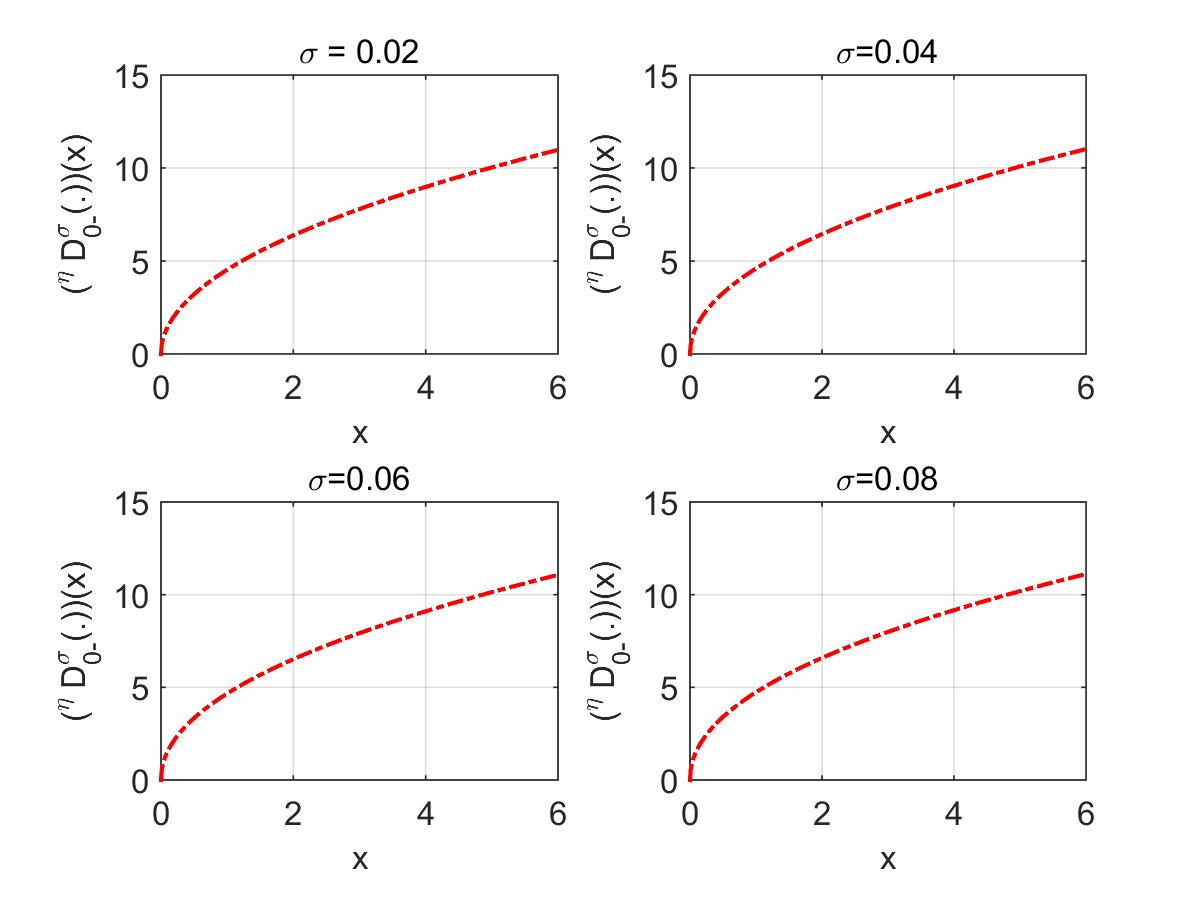}}
        \subfigure[Plot of Equation \eqref{TH1-1}]{\includegraphics[width=7.8 cm,height=7.6cm]{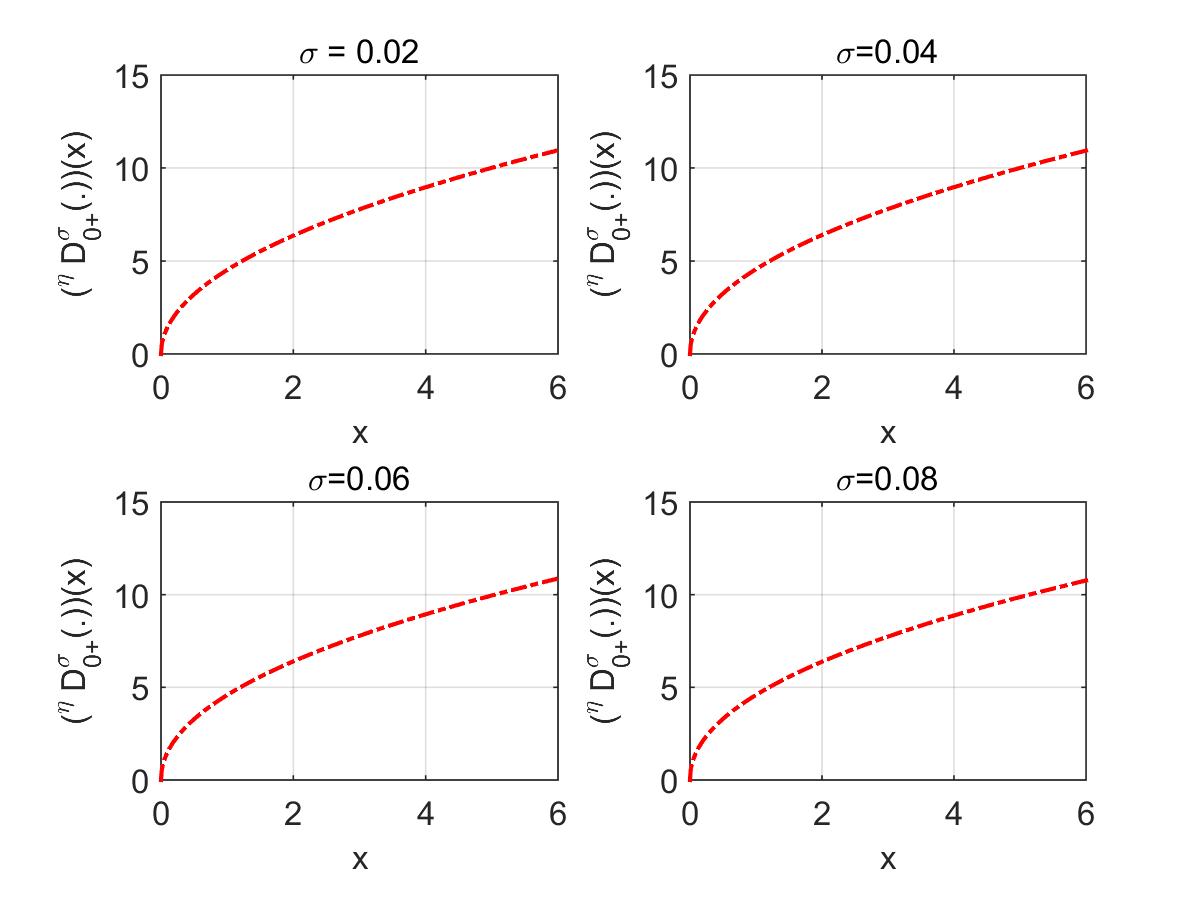}}}
        \caption{Graph  for $\eta=0.3;\sigma=0.02:0.02:0.08$} \label{fig-4}
    \end{figure}

\begin{figure}[htb]
    \centering
       \fbox{\subfigure[Plot of Equation \eqref{TH2-1}]{\includegraphics[width=7.8 cm,height=7.6cm]{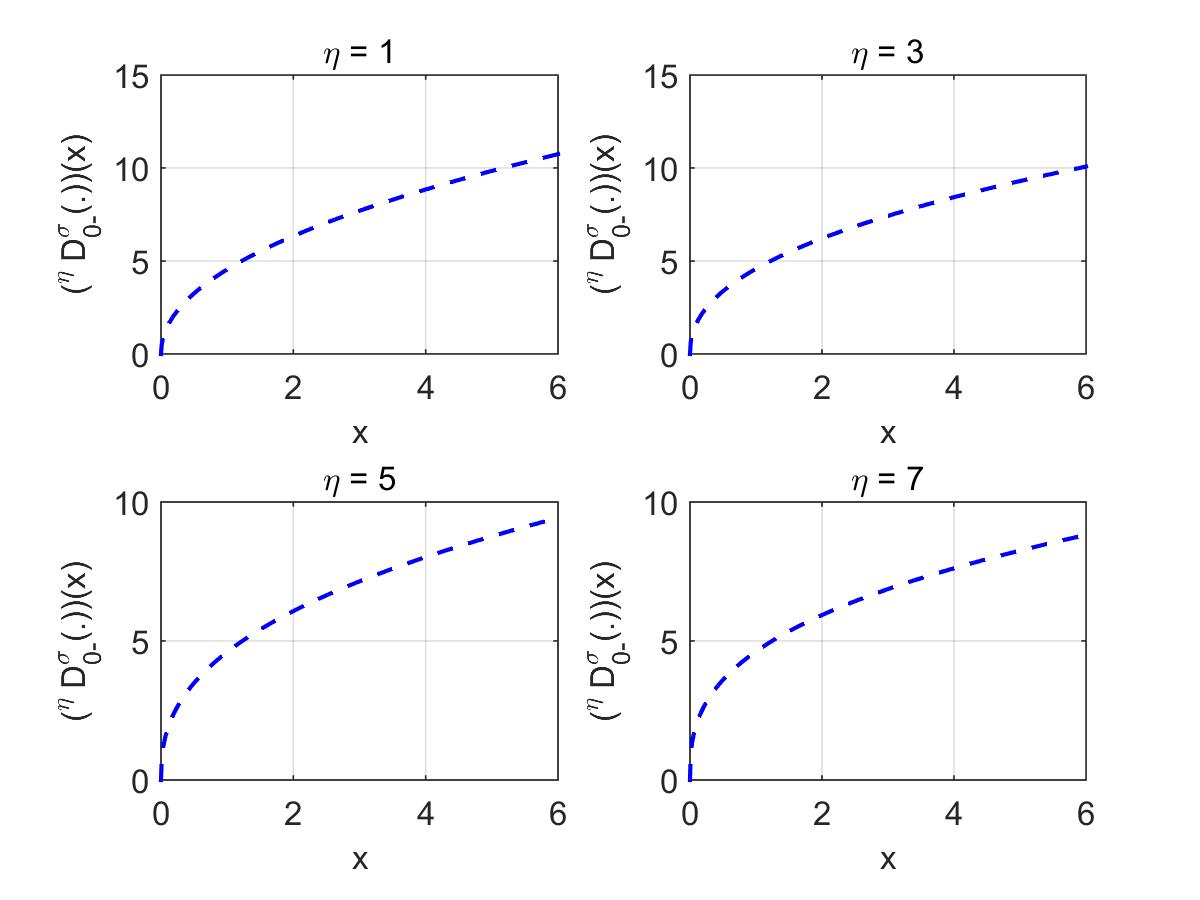}}
        \subfigure[Plot of Equation \eqref{TH1-1}]{\includegraphics[width=7.8 cm,height=7.6cm]{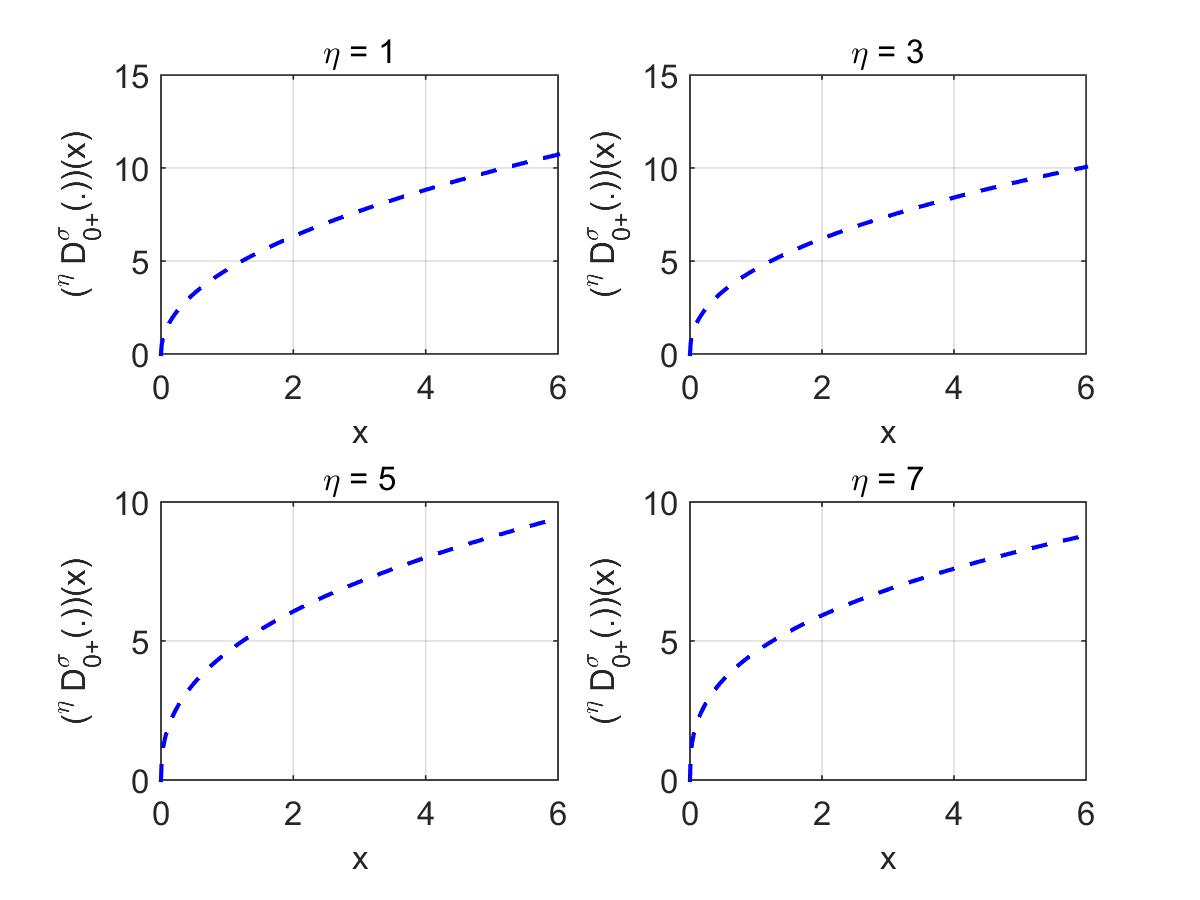}}}
        \caption{Graph  for $\sigma=0.02;\eta=1:2:7$} \label{fig-5}
    \end{figure}

\begin{table}[ht]\caption{Numerical Values of the Equation \eqref{TH1-1}}\label{table-1}
 \centering
\setlength{\tabcolsep}{11pt}
\renewcommand{\arraystretch}{1}
\begin{tabular}{ccccc}
\hline
 $x$ & $\sigma=0.1$ & $\sigma=0.2$ & $\sigma=0.3$ & $\sigma=0.4$ \\
\hline
 		0.00  &        0.00&          0.00  &        0.00  &        0.00\\
          0.50 &         3.30&          3.09&          2.48 &         1.44\\
          1.00&          4.57&          4.20&          3.30 &         1.88\\
          1.50&          5.53&          5.02&          3.90 &         2.19\\
          2.00&          6.33&          5.70&          4.38&          2.44\\
          2.50&          7.03&          6.28&          4.80&          2.66\\
          3.00&          7.66&          6.81&          5.18&          2.85\\
          3.50&          8.24&          7.29&          5.51&          3.02\\
          4.00&          8.77&          7.73&          5.82&          3.18\\
          4.50&          9.27&          8.14&          6.11&          3.33\\
          5.00&          9.74&          8.52&          6.38&          3.46\\
          5.50&         10.19&          8.89&          6.64&          3.59\\
          6.00&         10.61&          9.24&          6.88&          3.71\\
          6.50&         11.02&          9.57&          7.11&          3.83\\
          7.00&         11.41&          9.88&          7.32&          3.94\\
          7.50&         11.79&         10.19&          7.54&          4.04\\
          8.00&         12.15&         10.48&          7.74&          4.14\\
          8.50&         12.50&         10.77&          7.93&          4.24\\
          9.00&         12.84&         11.04&          8.12&          4.33\\
          9.50&         13.17&         11.31&          8.30&          4.42\\
         10.00&         13.49&         11.56&          8.48&          4.51\\
\hline
\end{tabular}
\end{table}

\begin{table}[ht]\caption{Numerical Values of the Equation \eqref{TH2-1}}\label{table-2}
 \centering
\setlength{\tabcolsep}{11pt}
\renewcommand{\arraystretch}{1}
\begin{tabular}{ccccc}
\hline
$ x$ & $\sigma=0.1$ & $\sigma=0.2$ & $\sigma=0.3$ & $\sigma=0.4$ \\
\hline
 0 &            0&             0&             0&             0\\
          0.50&          3.47&          3.83&          4.22&          4.67\\
          1.00&          4.81&          5.19&          5.61&          6.08\\
          1.50&          5.82&          6.20&          6.63&          7.09\\
          2.00&          6.66&          7.04&          7.46&          7.91\\
          2.50&          7.40&          7.77&          8.17&          8.61\\
          3.00&          8.06&          8.42&          8.80&          9.23\\
          3.50&          8.66&          9.01&          9.38&          9.79\\
          4.00&          9.22&          9.55&          9.91&         10.30\\
          4.50&          9.75&         10.06&         10.40&         10.77\\
          5.00&         10.24&         10.54&         10.86&         11.21\\
          5.50&         10.71&         10.99&         11.29&         11.62\\
          6.00&         11.16&         11.42&         11.70&         12.01\\
          6.50&         11.59&         11.83&         12.09&         12.38\\
          7.00&         12.00&         12.22&         12.46&         12.73\\
          7.50&         12.39&         12.59&         12.82&         13.07\\
          8.00&         12.78&         12.96&         13.16&         13.40\\
          8.50&         13.15&         13.31&         13.49&         13.71\\
          9.00&         13.50&         13.65&         13.81&         14.01\\
          9.50&         13.85&         13.97&         14.12&         14.30\\
         10.00&         14.19&         14.29&         14.42&         14.58\\
\hline
\end{tabular}
\end{table}

\begin{theorem}\label{Th3} Let $x>0, \sigma, \mu, \eta, \xi,\zeta,\vartheta  \in \mathbb{C}, k \in\mathbb{R}$, $\Re(\sigma)>0, \Re(\zeta)>0, \Re(\vartheta)>0,\Re(\mu)>0, \Re (l)>0,\Re (m)>0>$ and $q\in\mathbb{R}^+$, such that  $\eta>0$, then

\begin{eqnarray}\label{TH3-1}\aligned& B\left\{\left({}^{\eta}D^{\sigma}_{0+}t^\mu E^{\vartheta, q}_{k,\xi,\zeta}(tz)^\nu\right)(x) : l, m \right\}\\&=\Gamma(m)\hskip2mmx^{\mu-\sigma\eta}\frac{{\eta}^\sigma}{\Gamma\left(\frac{\vartheta}{k}\right)}k^{1-\frac{\zeta}{k}}{}_{3}\Psi_{3}\left[
 \begin{array}{cc} 
\left(\frac{\vartheta}{k},{q}\right), \left(\frac{\mu}{\eta}+1,\frac{\nu}{\eta}\right),(l,\nu),\\
	\left(\frac{\zeta}{k},\frac{\xi}{k}\right), \left(\frac{\mu}{\eta}-\sigma+1,\frac{\nu}{\eta},\right),
	(l+m,\nu)\end{array}\left| \hskip 1mm k^{(q-\frac{\xi}{k})}x^{\nu}\right
.\right].\endaligned \end{eqnarray}
\begin{proof}
For convenience, we denote the left-hand side of the result \eqref{TH3-1} by $\mathscr{B}$. Using the definition of Beta transform, the LHS of 
\eqref{TH3-1} becomes:

\begin{eqnarray}\label{TH3-2} \aligned& \mathscr{B}=\int_{0}^{1} z^{l-1} (1-z)^{m-1} \left({}^{\eta}D^{\sigma}_{0+}t^\mu E^{\vartheta, q}_{k,\xi,\zeta}(tz)^\nu\right)(x) dz,  \endaligned \end{eqnarray}

further using \eqref{k-ML} and then changing the order of integration and summation,which is valid under the conditions of Theorem 1, then 

\begin{eqnarray}\label{TH3-3}\aligned& \mathscr{B}=\sum_{r=0}^{\infty}\frac{(\vartheta)_{rq,k}}{\Gamma_{k}(r\xi+\zeta)}\frac{1}{r!}\left({}^{\eta}D^{\sigma}_{0+}t^{r\nu+\mu}\right)(x)\int_{0}^{1}{z^{l+r \nu-1}}(1-z)^{m-1}dz, \endaligned \end{eqnarray}

applying the result \eqref{d1}, after simplification Eq.\eqref{TH3-3} reduced to 

\begin{eqnarray}\label{TH3-4}\aligned& \mathscr{B}=x^{\mu-\sigma\eta}\frac{{\eta}^\sigma}{\Gamma\left(\frac{\vartheta}{k}\right)}k^{1-\frac{\zeta}{k}}\sum_{r=0}^{\infty}\frac{1}{r!}\frac{\Gamma(\frac{\vartheta}{k}+rq)}{\Gamma(\frac{\zeta}{k}+\frac{\xi}{k}r)} \frac{\Gamma\left(\frac{\mu}{\eta}+1+\frac{\nu}{\eta}r\right)}{\Gamma\left(\frac{\mu}{\eta}+1-\sigma+\frac{\nu}{\eta}r\right)}  k^{r(q-\frac{\xi}{k})} \int_{0}^{1}{z^{l+n \nu-1}}(1-z)^{m-1}dz, \endaligned \end{eqnarray}

applying the definition of Beta transform, Eq.\eqref{TH3-4} reduced to

\begin{eqnarray}\label{TH3-5}\aligned& \mathscr{B}=x^{\mu-\sigma\eta}\frac{{\eta}^\sigma}{\Gamma\left(\frac{\vartheta}{k}\right)}k^{1-\frac{\zeta}{k}}\sum_{r=0}^{\infty}\frac{1}{r!}\frac{\Gamma(\frac{\vartheta}{k}+rq)}{\Gamma(\frac{\zeta}{k}+\frac{\xi}{k}r)} \frac{\Gamma\left(\frac{\mu}{\eta}+1+\frac{\nu}{\eta}r\right)}{\Gamma\left(\frac{\mu}{\eta}+1-\sigma+\frac{\nu}{\eta}r\right)}  k^{r(q-\frac{\xi}{k})}\frac{\Gamma(l+\nu n) \Gamma(m)}{\Gamma(l+m+\nu n)}. \endaligned \end{eqnarray}

\begin{eqnarray}\label{TH3-6}\aligned& \mathscr{B}=\Gamma(m)\hskip2mmx^{\mu-\sigma\eta}\frac{{\eta}^\sigma}{\Gamma\left(\frac{\vartheta}{k}\right)}k^{1-\frac{\zeta}{k}}{}_{3}\Psi_{3}\left[
 \begin{array}{cc} 
\left(\frac{\vartheta}{k},{q}\right), \left(\frac{\mu}{\eta}+1,\frac{\nu}{\eta}\right),(l,\nu),\\
	\left(\frac{\zeta}{k},\frac{\xi}{k}\right), \left(\frac{\mu}{\eta}-\sigma+1,\frac{\nu}{\eta},\right),
	(l+m,\nu)\end{array}\left| \hskip 1mm k^{(q-\frac{\xi}{k})}x^{\nu}\right
.\right].\endaligned \end{eqnarray}

\end{proof}
\end{theorem}

\begin{theorem}\label{Th4} Let $x>0, \sigma, \mu, \eta, \xi,\zeta,\vartheta  \in \mathbb{C}, k \in\mathbb{R}$, $\Re(\sigma)>0, \Re(\zeta)>0, \Re(\vartheta)>0,\Re(\mu)>0, \Re (l)>0,\Re (m)>0>$ and $q\in\mathbb{R}^+$, such that  $\eta>0$, then

\begin{eqnarray}\label{TH4-1}\aligned& B\left\{\left({}^{\eta}D^{\sigma}_{0-}t^\mu E^{\vartheta, q}_{k,\xi,\zeta}(tz)^\nu\right)(x) : l, m \right\}=(-1)^{-\sigma}\Gamma(m)\hskip2mm x^{\mu-\sigma\eta}\frac{{\eta}^\sigma}{\Gamma\left(\frac{\vartheta}{k}\right)}k^{1-\frac{\zeta}{k}}\\&\hskip 12mm\times{}_{3}\Psi_{3}\left[
 \begin{array}{cc} 
\left(\frac{\vartheta}{k},{q}\right), \left(\frac{\mu}{\eta}+1,\frac{\nu}{\eta}\right),(l,\nu),\\
	\left(\frac{\zeta}{k},\frac{\xi}{k}\right), \left(\frac{\mu}{\eta}-\sigma+1,\frac{\nu}{\eta},\right),
	(l+m,\nu)\end{array}\left| \hskip 1mm k^{(q-\frac{\xi}{k})}x^{\nu}\right
.\right].\endaligned \end{eqnarray}

\begin{proof}
For convenience, we denote the left-hand side of the result \eqref{TH4-1} by $\mathscr{B}$. Using the definition of Beta transform, the LHS of 
\eqref{TH4-1} becomes:

\begin{eqnarray}\label{TH4-2} \aligned& \mathscr{B}=\int_{0}^{1} z^{l-1} (1-z)^{m-1} \left({}^{\eta}D^{\sigma}_{0-}t^\mu E^{\vartheta, q}_{k,\xi,\zeta}(tz)^\nu\right)(x) dz,  \endaligned \end{eqnarray}

further using \eqref{k-ML} and then changing the order of integration and summation,which is valid under the conditions of Theorem 2, then 

\begin{eqnarray}\label{TH4-3}\aligned& \mathscr{B}=\sum_{n=0}^{\infty}\frac{(\vartheta)_{rq,k}}{\Gamma_{k}(r\xi+\zeta)}\frac{1}{r!}\left({}^{\eta}D^{\sigma}_{0-}t^{r\nu+\mu}\right)(x)\int_{0}^{1}{z^{l+r\nu-1}}(1-z)^{m-1}dz, \endaligned \end{eqnarray}

applying the result \eqref{d2}, after simplification Eq.\eqref{TH4-3} reduced to 

\begin{eqnarray}\label{TH4-4}\aligned& \mathscr{B}=(-1)^{-\sigma} \hskip 1mm x^{\mu-\sigma\eta}\frac{{\eta}^\sigma}{\Gamma\left(\frac{\vartheta}{k}\right)}k^{1-\frac{\zeta}{k}}\sum_{r=0}^{\infty}\frac{1}{r!}\frac{\Gamma(\frac{\vartheta}{k}+rq)}{\Gamma(\frac{\zeta}{k}+\frac{\xi}{k}r)} \frac{\Gamma\left(\frac{\mu}{\eta}+1+\frac{\nu}{\eta}r\right)}{\Gamma\left(\frac{\mu}{\eta}+1-\sigma+\frac{\nu}{\eta}r\right)}\\&\hskip 26mm\times  k^{r(q-\frac{\xi}{k})} \int_{0}^{1}{z^{l+n \nu-1}}(1-z)^{m-1}dz, \endaligned \end{eqnarray}

applying the definition of Beta transform, Eq.\eqref{TH4-4} reduced to

\begin{eqnarray}\label{TH4-5}\aligned& \mathscr{B}=(-1)^{-\sigma} \hskip 1mm x^{\mu-\sigma\eta}\frac{{\eta}^\sigma}{\Gamma\left(\frac{\vartheta}{k}\right)}k^{1-\frac{\zeta}{k}}\sum_{r=0}^{\infty}\frac{1}{r!}\frac{\Gamma(\frac{\vartheta}{k}+rq)}{\Gamma(\frac{\zeta}{k}+\frac{\xi}{k}r)} \frac{\Gamma\left(\frac{\mu}{\eta}+1+\frac{\nu}{\eta}r\right)}{\Gamma\left(\frac{\mu}{\eta}+1-\sigma+\frac{\nu}{\eta}r\right)}\\&\hskip45mm\times  k^{r(q-\frac{\xi}{k})}\frac{\Gamma(l+\nu n) \Gamma(m)}{\Gamma(l+m+\nu n)}. \endaligned \end{eqnarray}

\begin{eqnarray}\label{TH4-6}\aligned& \mathscr{B}=(-1)^{-\sigma} \hskip 1mm\Gamma(m)\hskip2mmx^{\mu-\sigma\eta}\frac{{\eta}^\sigma}{\Gamma\left(\frac{\vartheta}{k}\right)}k^{1-\frac{\zeta}{k}}{}_{3}\Psi_{3}\left[
 \begin{array}{cc} 
\left(\frac{\vartheta}{k},{q}\right), \left(\frac{\mu}{\eta}+1,\frac{\nu}{\eta}\right),(l,\nu),\\
	\left(\frac{\zeta}{k},\frac{\xi}{k}\right), \left(\frac{\mu}{\eta}-\sigma+1,\frac{\nu}{\eta},\right),
	(l+m,\nu)\end{array}\left| \hskip 1mm k^{(q-\frac{\xi}{k})}x^{\nu}\right
.\right].\endaligned \end{eqnarray}

\end{proof}
\end{theorem}

\begin{theorem}\label{Th5} Let $x>0, \sigma, \mu, \eta, \xi,\zeta,\vartheta  \in \mathbb{C}, k \in\mathbb{R}$, $\Re(\sigma)>0, \Re(\zeta)>0, \Re(\vartheta)>0,\Re(\mu)>0, \Re (l)>0,\Re (m)>0>$ and $q\in\mathbb{R}^+$, such that  $\eta>0$, then

\begin{eqnarray}\label{TH5-1}\aligned& L\left\{z^{l-1}\left({}^{\eta}D^{\sigma}_{0+}t^\mu E^{\vartheta, q}_{k,\xi,\zeta}(tz)^\nu\right)(x)\right\}\\&=\frac{x^{\mu-\sigma\eta}}{s^l}\frac{{\eta}^\sigma}{\Gamma\left(\frac{\vartheta}{k}\right)}k^{1-\frac{\zeta}{k}}{}_{3}\Psi_{2}\left[
 \begin{array}{cc} 
\left(\frac{\vartheta}{k},\frac{q}{k}\right), \left(\frac{\mu}{\eta}+1,\frac{\nu}{\eta}\right),(l,\nu)\\
	\left(\frac{\zeta}{k},\frac{\xi}{k}\right), \left(\frac{\mu}{\eta}-\sigma+1,\frac{\nu}{\eta}\right) \end{array}\left|k^{(q-\frac{\xi}{k})}(\frac{x}{s})^{\nu}\right.
\right].\endaligned  \end{eqnarray}
\begin{proof}
For convenience, we denote the left-hand side of the result \eqref{TH5-1} by $\mathscr{L}$. Using the definition of Laplace transform, the LHS of 
\eqref{TH5-1} becomes:

\begin{eqnarray}\label{TH5-2} \aligned& \mathscr{L}=\int_{0}^{\infty} e^{-sz}z^{l-1}\left({}^{\eta}D^{\sigma}_{0+}t^\mu E^{\vartheta, q}_{k,\xi,\zeta}(tz)^\nu\right)(x) dz, \endaligned \end{eqnarray}

further using \eqref{k-ML} and then changing the order of integration and summation,which is valid under the conditions of Theorem 1, then 
applying the result \eqref{d1}, after simplification Eq.\eqref{TH5-2} reduced to

\begin{eqnarray}\label{TH5-3}\aligned& \mathscr{L}=\sum_{r=0}^{\infty}\frac{(\vartheta)_{rq,k}}{\Gamma_{k}(r\xi+\zeta)}\frac{1}{r!}\left({}^{\eta}D^{\sigma}_{0+}t^{r\nu+\mu}\right)(x)\int_{0}^{\infty}{e^{-sz}z^{l+r \nu-1}}dz. \endaligned \end{eqnarray}

Eq.\eqref{TH5-3} reduced to

\begin{eqnarray}\label{TH5-4}\aligned& \mathscr{L}=x^{\mu-\sigma\eta}\frac{{\eta}^\sigma}{\Gamma\left(\frac{\vartheta}{k}\right)}k^{1-\frac{\zeta}{k}}\sum_{r=0}^{\infty}\frac{1}{r!}\frac{\Gamma(\frac{\vartheta}{k}+rq)}{\Gamma(\frac{\zeta}{k}+\frac{\xi}{k}r)} \frac{\Gamma\left(\frac{\mu}{\eta}+1+\frac{\nu}{\eta}r\right)}{\Gamma\left(\frac{\mu}{\eta}+1-\sigma+\frac{\nu}{\eta}r\right)}  k^{r(q-\frac{\xi}{k})}\times\frac{\Gamma(l+\nu r)}{s^{\l+\nu r}}{}. \endaligned \end{eqnarray}

\begin{eqnarray}\label{TH5-5}\aligned& \mathscr{L}=\frac{x^{\mu-\sigma\eta}}{s^l}\frac{{\eta}^\sigma}{\Gamma\left(\frac{\vartheta}{k}\right)}k^{1-\frac{\zeta}{k}}\times{}_{3}\Psi_{2}\left[
 \begin{array}{cc} 
\left(\frac{\vartheta}{k},\frac{q}{k}\right), \left(\frac{\mu}{\eta}+1,\frac{\nu}{\eta}\right),(l,\nu)\\
	\left(\frac{\zeta}{k},\frac{\xi}{k}\right), \left(\frac{\mu}{\eta}-\sigma+1,\frac{\nu}{\eta}\right) \end{array}\left|k^{(q-\frac{\xi}{k})}(\frac{x}{s})^{\nu}\right.
\right].\endaligned  \end{eqnarray}

\end{proof}

\end{theorem}

\begin{theorem}\label{Th6} Let $x>0, \sigma, \mu, \eta, \xi,\zeta,\vartheta  \in \mathbb{C}, k \in\mathbb{R}$, $\Re(\sigma)>0, \Re(\zeta)>0, \Re(\vartheta)>0,\Re(\mu)>0, \Re (l)>0,\Re (m)>0>$ and $q\in\mathbb{R}^+$, such that  $\eta>0$, then

\begin{eqnarray}\label{TH6-1}\aligned& L\left\{z^{l-1}\left({}^{\eta}D^{\sigma}_{0-}t^\mu E^{\vartheta, q}_{k,\xi,\zeta}(tz)^\nu\right)(x)\right\}=(-1)^{-\sigma}\frac{x^{\mu-\sigma\eta}}{s^l}\frac{{\eta}^\sigma}{\Gamma\left(\frac{\vartheta}{k}\right)}k^{1-\frac{\zeta}{k}}\\&\hskip6mm\times{}_{3}\Psi_{2}\left[
 \begin{array}{cc} 
\left(\frac{\vartheta}{k},\frac{q}{k}\right), \left(\frac{\mu}{\eta}+1,\frac{\nu}{\eta}\right),(l,\nu)\\
	\left(\frac{\zeta}{k},\frac{\xi}{k}\right), \left(\frac{\mu}{\eta}-\sigma+1,\frac{\nu}{\eta}\right) \end{array}\left|k^{(q-\frac{\xi}{k})}(\frac{x}{s})^{\nu}\right.
\right].\endaligned  \end{eqnarray}
\begin{proof}
For convenience, we denote the left-hand side of the result \eqref{TH6-1} by $\mathscr{L}$. Using the definition of Laplace transform, the LHS of 
\eqref{TH5-1} becomes:

\begin{eqnarray}\label{TH6-2} \aligned& \mathscr{L}=\int_{0}^{\infty} e^{-sz}z^{l-1}\left({}^{\eta}D^{\sigma}_{0-}t^\mu E^{\vartheta, q}_{k,\xi,\zeta}(tz)^\nu\right)(x) dz, \endaligned \end{eqnarray}

further using \eqref{k-ML} and then changing the order of integration and summation,which is valid under the conditions of Theorem 2, then 
applying the result \eqref{d2}, after simplification Eq.\eqref{TH6-2} reduced to

\begin{eqnarray}\label{TH6-3}\aligned& \mathscr{L}=(-1)^{-\sigma}\sum_{r=0}^{\infty}\frac{(\vartheta)_{rq,k}}{\Gamma_{k}(r\xi+\zeta)}\frac{1}{r!}\left({}^{\eta}D^{\sigma}_{0+}t^{r\nu+\mu}\right)(x)\int_{0}^{\infty}{e^{-sz}z^{l+r \nu-1}}dz. \endaligned \end{eqnarray}

Eq.\eqref{TH6-3} reduced to

\begin{eqnarray}\label{TH6-4}\aligned& \mathscr{L}=(-1)^{-\sigma}x^{\mu-\sigma\eta}\frac{{\eta}^\sigma}{\Gamma\left(\frac{\vartheta}{k}\right)}k^{1-\frac{\zeta}{k}}\sum_{r=0}^{\infty}\frac{1}{r!}\frac{\Gamma(\frac{\vartheta}{k}+rq)}{\Gamma(\frac{\zeta}{k}+\frac{\xi}{k}r)} \frac{\Gamma\left(\frac{\mu}{\eta}+1+\frac{\nu}{\eta}r\right)}{\Gamma\left(\frac{\mu}{\eta}+1-\sigma+\frac{\nu}{\eta}r\right)}  k^{r(q-\frac{\xi}{k})}\frac{\Gamma(l+\nu r)}{s^{\l+\nu r}}{}. \endaligned \end{eqnarray}

\begin{eqnarray}\label{TH6-5}\aligned& \mathscr{L}=(-1)^{-\sigma}\frac{x^{\mu-\sigma\eta}}{s^l}\frac{{\eta}^\sigma}{\Gamma\left(\frac{\vartheta}{k}\right)}k^{1-\frac{\zeta}{k}}{}_{3}\Psi_{2}\left[
 \begin{array}{cc} 
\left(\frac{\vartheta}{k},\frac{q}{k}\right), \left(\frac{\mu}{\eta}+1,\frac{\nu}{\eta}\right),(l,\nu)\\
	\left(\frac{\zeta}{k},\frac{\xi}{k}\right), \left(\frac{\mu}{\eta}-\sigma+1,\frac{\nu}{\eta}\right) \end{array}\left|k^{(q-\frac{\xi}{k})}(\frac{x}{s})^{\nu}\right.
\right].\endaligned  \end{eqnarray}

\end{proof}

\end{theorem}

\section*{References}
%
%\bibliography{mybibfile}
%\end{document}

\end{document}